\documentclass[12pt]{article}

\usepackage{xcolor}  
\usepackage{amsmath, amsthm, amsfonts, amssymb}
\newtheorem{thm}{Theorem}[section]
\newtheorem{lemma}[thm]{Lemma}
\theoremstyle{definition}
\newtheorem{Definicion}[thm]{Definition}
\newtheorem{example}[thm]{Example}

\newtheorem{cor}{Corollary}
\newtheorem{prop}{Proposition}
\theoremstyle{remark}
\newtheorem{remark}[thm]{Remark}
\numberwithin{equation}{section}
\title{Fourier analysis  with generalized integration}
\author{ Juan H. Arredondo, M. Guadalupe Morales, Manuel Bernal G.}
\date{}
\begin{document}
\maketitle
\begin{abstract}
 We generalize the classic  Fourier  transform operator $\mathcal{F}_{p}$ by using the Henstock-Kurzweil integral theory.  It is shown that the operator equals the $HK$-Fourier transform  on a  dense subspace of $\mathcal{ L}^p$, $1<p\leq 2$. In particular, a theoretical scope of this representation is raised to approximate  numerically the Fourier transform of functions on the mentioned subspace. Besides, we show differentiability of the Fourier transform function $\mathcal{F}_{p}(f)$ under more general conditions than in Lebesgue's theory.
 \medskip


                  Secondary 42B10, 26A42, 97N50, 26B30, 58C05.

 \smallskip


\end{abstract}

\section{Introduction}\label{intro}
The Lebesgue integral  has strong implications in Fourier Analysis. Integration theory had an important development in the last half-century. Thereby, with the introduction of new integration theories, the possibility  to extend fundamental results arises, allowing  new and better numerical approaches. For example, the Henstock-Kurzweil integral  contains Riemann, improper Riemann and Lebesgue integrals with  the values of the integrals  coinciding \cite{Swartz}. Thus, in \cite{M.Guadalupe}  was proved that for  subsets of $p$-integrable functions, $1<p\leq 2$, the classical Fourier transform can be represented as a Henstock-Kurzweil integral for any $s\neq 0$. 
Moreover, this representation allows to analyze more properties related  to the Fourier transform, as continuity or asymptotic behavior. 
On the other hand,  Fourier Analysis  is related  to Approximation theory. Important applications are based on integration theory \cite{Ruzhansky}.  

\medskip

 In  \cite{Liflyand} the Fourier transform is studied  for functions in $BV_0(\mathbb{R})$ whose
derivative lies in the Hardy space  $H^1(\mathbb{R}) \subsetneq \mathcal{L}^1(\mathbb{R})$, but not dense in it. This operator is known as the Fourier–Stieltjes transform.  In  \cite{Ivashev-Musatov}, \cite{Liflyand0} - \cite{Liflyand3}, \cite{Wiener} the  Fourier–Stieltjes transform was analyzed  obtaining asymptotic formulas and integrability for the Fourier Cosine and Sine transforms of such kind of functions.

\medskip

In this work we study the Fourier transform over subsets of $ \mathcal{L}^{p}(\mathbb{R})\setminus\mathcal{L}^{1}(\mathbb{R})$ when  $1<p\leq 2$.  Some classical results for the Fourier transform will be extended on subspaces of $BV_0(\mathbb{R})$ not contained in $\mathcal{L}^{1}(\mathbb{R})$.  
We will prove that  the classical Fourier transform $\mathcal{F}_{p}(f)$  equals the $HK$-Fourier transform,  $\mathcal{F}_{HK}(f)$, for $1<p\leq 2$. Thus, $\mathcal{F}_{p}(f)$ keeps the continuity property a.e. and asymptotic behavior. Moreover, the possibility to obtain a numerical approximation of $\mathcal{F}_{p}(f)$ is shown, via the Henstock-Kurzweil integral. In the last section, we will show that  the Fourier transform on $\mathcal{L}^p(\mathbb{R}) $ is differentiable under more general conditions implied by those in Lebesgue's theory  \cite[Theorem 9.2]{Rudin}.




\section{Preliminaries} \label{sec:1}
We follow the notation from \cite{Bartle}  to introduce basic definitions of the Henstock-Kurzweil integral.
Let $\mathbb{\overline{R}}=\mathbb{R}\cup\lbrace\pm \infty\rbrace$ and  $[a,b]$ be a non-degenerative  interval in $\mathbb{\overline{R}}$. A partition  $P$ of  $[a,b]$  is a finite collection of  non-overlapping  intervals  such that 
\begin{equation*}
[a,b]=I_{1}\cup I_2\cup \ldots \cup I_{n}.
\end{equation*}  
Specifically, the partition itself is the set of endpoints of each  sub-interval $I_i$.
\begin{equation*}
a=x_{0}\leq x_{1}\leq x_{2}\leq \ldots \leq x_{n}=b,
\end{equation*}
where  
\begin{equation*}
\begin{array}{crl}
I_{i}:=[x_{i-1},x_{i}] &  for  & i=1,\ldots,n,
\end{array}
\end{equation*} Observe that $[a,b]$ can be unbounded.


\begin{Definicion}
A tagged partition of  $I$ 
\begin{equation*}
\dot{P}=\left\lbrace(I_{i},t_{i})\right\rbrace_{i=1}^{n}.
\end{equation*}
is a finite set of ordered pairs  $\left\lbrace(I_{i},t_{i} )\right\rbrace_{i=1}^{n}$, where the collection of subintervals $\{I_{i}\}$ forms a partition of $I$ and the point $t_{i}\in I_i$ is called a tag of $I_i$.
\end{Definicion}


\begin{Definicion}
A map $\delta:[a,b]\rightarrow (0,\infty)$ is  called gauge function on $[a,b]$. Given a gauge function $\delta$ on $[a,b]$ it is said that a tagged partition 
$\dot{P}=\left\lbrace([x_{i-1}, x_{i}]; t_i)\right\rbrace_{i=1}^{n}$ of $[a,b]$ is $\delta$-fine according to the following cases:\\
For $a \in \mathbb{R}$ and $b=\infty$: \begin{enumerate}
  \item $a=x_{0},b=x_{n}=t_{n}=\infty$.
  \item $[x_{i-1}, x_{i}]\subset [t_{i}-\delta (t_{i}), t_{i} + \delta(t_{i})]$, for all $i=1,2,\ldots  ,n-1$.
  \item $[x_{n-1},\infty]\subset [\frac{1}{\delta(t_{n})},\infty]$. 
\end{enumerate}
For $a=-\infty$ and $b\in \mathbb{R}$:
\begin{enumerate}
  \item $a=x_{0}=t_{1}=-\infty$, $b=x_{n}$.
  \item $[x_{i-1}, x_{i}]\subset[t_{i}-\delta (t_{i}), t_{i} + \delta(t_{i})]$, for all $i=2,\ldots,n$.
  \item $[-\infty,x_{1}]\subset [-\infty,-\frac{1}{\delta(t_{1})}]$. 
\end{enumerate}
For $a=-\infty$ and $b=\infty$
\begin{enumerate}
  \item $a=x_{0}=t_{1}=-\infty$, $b=x_{n}=t_{n}=\infty$.
  \item $[x_{i-1}, x_{i}]\subset [t_{i}-\delta (t_{i}), t_{i} + \delta(t_{i})]$, for all $i=2,\ldots,  n-1$.
  \item $[x_{n-1},\infty]\subset [\frac{1}{\delta(t_{n})},\infty]$ and $[-\infty,x_{1}]\subset [-\infty,-\frac{1}{\delta(t_{1})}]$. 
\end{enumerate}
For $a,b\in \mathbb{R}$:
\begin{enumerate}
\item $[x_{i-1},x_{i}]\subset [t_{i}-\delta (t_{i}), t_{i} + \delta(t_{i})]$, for all $i=1,2,\ldots,n.$
\end{enumerate}
\end{Definicion}

According to the convention concerning to the `` arithmetic " in $\overline{\mathbb{R}}$, $0\cdot(\pm\infty)=0$, a real-valued function  $f$ defined over $\mathbb{R}$ can be  extended by setting $f(\pm\infty)=0$. Thus, it is  introduced the definition of the  Henstock-Kurzweil integral over intervals in $\overline{\mathbb{R}}$.

\begin{Definicion}
Let $[a,b]$ be an interval in $\overline{\mathbb{R}}$. The real-valued function $f$ defined over $I$ is said to be  Henstock-Kurzweil integrable on $[a,b]$ iff  there exists $A \in\mathbb{R}$ such that for every $\epsilon > 0$ there exists a gauge function $\delta_{\epsilon}$ over $[a,b]$, such that if  $\dot{P}= \left\lbrace([x_{i-1}, x_{i}]; t_i)\right\rbrace_{i=1}^{n}$ is a $\delta_{\epsilon}$-fine   partition of $[a,b]$, then
\begin{equation*}
\left| \sum_{i=1}^{n}f(t_{i})(x_{i}-x_{i-1})-A\right|<\epsilon.
\end{equation*}
The number $A$ is the integral of  $f$ over $[a,b]$ and it is denoted by $\int_{a}^{b}f =A$. 
\end{Definicion}
The set of all Henstock-Kurzweil integrable functions on the interval $I$ is denoted by $ HK(I)$, and the set of Henstock-Kurzweil integrable functions over each compact interval is denoted by $HK_{loc}(\mathbb{R})$.
The integrals will be in the Henstock-Kurzweil sense, if not specified.

 The  Multiplier Theorem in \cite{Bartle} states that the bounded variation functions are the multipliers of the  Henstock-Kurzweil integrable functions. Moreover, this concept is related to the Riemann-Stieltjes integral, which  generalizes the Riemann integral and also it is useful to calculate the Fourier transform, see Theorem \ref{spaceLoc}  below. There exist several versions of the  Riemann-Stieltjes integral. In this work it is considered the Riemann-Stieljes ($\delta$)-integral, also called norm Riemann-Stieltjes integral \cite{Mc,Monteiro}. Here for simplicity, it is called as Riemann-Stieltjes integral.


\medskip

The set of bounded variation functions over $I\subseteq \mathbb{R}$ is denoted by $BV(I)$, 
 and  $BV_{0}(\mathbb{R})$ denotes the functions in $BV(\mathbb{R})$  vanishing  at infinity, \cite{Liflyand,Mendoza}.

\subsection{The Fourier transform operator in the classical sense} \label{sec:2}
We will enunciate basic results about the Fourier  Analysis in the classical sense, i.e. conside\-ring Lebesgue's integral.
\begin{Definicion}\label{normaLP}
Let $1\leq p<\infty$. For any Lebesgue measurable function $f:\mathbb{R}\rightarrow \mathbb{R}$ we define
\begin{equation*}
\Vert f \Vert_{p}=\left[\int_{\mathbb{R}} |f|^{p}d\mu \right]^{\frac{1}{p}}.
\end{equation*}

\end{Definicion}
For each $p\geq 1$, the set of functions $f$ such that $\Vert f \Vert_{p}<\infty$ (called $p$-integrable functions) is a normed space (considering  equivalence classes respect to $\Vert \cdot \Vert_{p}$) and is denoted by $\mathcal{L}^{p}(\mathbb{R} )$.


The Fourier transform has been developed in the context of the Lebesgue theory and over the spaces $\mathcal{L}^{p}(\mathbb{R})$, with important implications in different areas, e.g. optics, signal theory, statistics, probability theory \cite{Bracewell}.


The set $\mathcal{L}^{1}(\mathbb{R})$ is well known as the space of Lebesgue integrable functions on $\mathbb{R}$ or absolutely integrable functions.  $\mathcal{L}^{2}(\mathbb{R})$ is a normed space with an inner product  that provides algebraic and geometric techniques applicable to spaces of arbitrary dimension.  These spaces are usually  considered  to define the  Fourier transform \cite{Grafakos,Reed, Rudin}.

\begin{Definicion}\label{TransClas}
Let $f \in \mathcal{L}^1(\mathbb{R})$. The Fourier transform  of $f$ at the point $s$  is defined as 
\begin{equation*}
\mathcal{F}_{1}(f)(s)=\frac{1}{\sqrt{2\pi}}\int_{-\infty}^{\infty}e^{-is x} f(x)dx.
\end{equation*}

\end{Definicion}

It is well known that $\mathcal{F}_{1}(f)$ is pointwise defined and by the Riemann-Lebesgue Lemma,  $\mathcal{F}_{1}(f)$ belongs to $C_\infty(\mathbb{R})$, the set of complex-valued continuous functions on $\mathbb{R}$  vanishing at $\pm\infty$  \cite{M.Guadalupe, Reed}.
On the other hand, the space $\mathcal{L}^2(\mathbb{R})$ is not contained in $\mathcal{L}^{1}(\mathbb {R})$. Thus, the operator $\mathcal{F}_{1}$ is not well defined over $\mathcal{L}^2(\mathbb{R})$. Nevertheless, the Fourier transform on  $\mathcal{L}^{2}(\mathbb{R})$ is given as an extension of the Fourier transform $\mathcal{F}_{1}(f)$ initially defined on  
 $\mathcal{L}^{1}(\mathbb {R})\cap\mathcal{ L}^{2}(\mathbb{R})$. It means that the Fourier transform over $\mathcal{L}^2(\mathbb{R})$ is defined as a limit  \cite{Hilbert}, and is denoted as $\mathcal{F}_2$.

\medskip

\begin{Definicion}[\cite{Grafakos, Pinsky,Reed}] The Fourier transform operator  in $\mathcal{L}^p(\mathbb{R})$, $1<p<2$, is given as $$\mathcal{F}_p:\mathcal{L}^p(\mathbb{R})\rightarrow\mathcal{L}^q(\mathbb{R})$$
$$\mathcal{F}_p(f)=\mathcal{F}_1(f_1)+\mathcal{F}_2(f_2),$$
where $p^{-1}+q^{-1}=1$, $f_1\in\mathcal{L}^1(\mathbb{R})\cap\mathcal{L}^p(\mathbb{R})$, $f_2\in\mathcal{L}^2(\mathbb{R})\cap\mathcal{L}^p(\mathbb{R})$ and $f=f_1+f_2$.
\end{Definicion}

\section{ $HK$-Fourier transform}  \label{sec:3}

At the beginning of this century, the  Fourier theory was developed using  the  Henstock-Kurzweil theory. In \cite{Erick}, E. Talvila   showed some existence theorems and continuity of the   Fourier transform over Henstock-Kurzweil integrable functions. Moreover, the  Fourier transform has been studied as a Henstock-Kurzweil integral over non-classical spaces of functions \cite{Francisco, M.Guadalupe}.
It is well known that if $I$ is a compact interval, then
\begin{equation*}
 BV(I)\subset \mathcal{L}^{1}(I)\subset HK(I).
 \end{equation*}
 However, when $I$ is an unbounded interval, 
\begin{equation*}\label{Conjunto1}
BV(I)\nsubseteq \mathcal{L}^1(I)
\end{equation*}
 and
\begin{equation*}\label{Conjunto2}
\mathcal{L}^{1}(I)\nsubseteq HK(I)\cap BV(I).
\end{equation*}

 Thereby, when $I$ is unbounded, there is no inclusion relation between $\mathcal{L}^{1}(I)$ and $HK(I)\cap BV(I)$. On the other hand,  $BV(I) \cap HK(I) \subset \mathcal{L}^{2}(I)$ (by the Multiplier Theorem  \cite{Bartle}). In \cite[Lemma 4.1]{Salvador} it is  proved that $BV(I) \cap HK(I)\subset BV_0(\mathbb{R})$. Accordingly, it is possible  to analyze the Fourier transform via the Henstock-Kurzweil integral over $BV_{0}(\mathbb{R})$.
 Thus, in \cite{Francisco} it was shown a generalized Riemann-Lebesgue lemma on unbounded intervals, giving rise to the definition of the $HK$-Fourier transform \cite{M.Guadalupe}. 

\begin{Definicion} 
$\mathcal{L}^{1}(\mathbb{R})+ BV_{0}(\mathbb{R})$ denotes the vector space of functions $f=f_{1}+f_{2}$, where $f_{1}\in \mathcal{L}^{1}(\mathbb{R})$ and $f_{2}\in BV_{0}(\mathbb{R})$.
\end{Definicion}

\begin{Definicion}\label{TransformadaHK}
The $HK$-Fourier transform  is defined as
\begin{equation*}
\mathcal{F}_{HK}:\mathcal{L}^{1}(\mathbb{R})+ BV_{0}(\mathbb{R})\rightarrow C_{\infty}(\mathbb{R}\setminus\{0\}),\end{equation*}
\begin{eqnarray}\label{TransformadaHK1}
\mathcal{F}_{HK}(f)(s)&:=&\frac{1}{\sqrt{2\pi}} \int_{\mathbb{R}}e^{-isx}f(x)dx\nonumber\\
&=& \frac{1}{\sqrt{2\pi}} \int_{\mathbb{R}}cos(sx)f(x)dx -i \frac{1}{\sqrt{2\pi}} \int_{\mathbb{R}}sin(sx)f(x)dx\\
&\equiv & \mathcal{F}^{C}_{HK}(f)(s)-i \mathcal{F}^{S}_{HK}(f)(s)\nonumber
\end{eqnarray}
where the integrals are in Henstock-Kurzweil sense. $\mathcal{F}^{C}_{HK}(f)$ and  $\mathcal{F}^{S}_{HK}(f)$ are called the $HK$-Cosine Fourier and the $HK$-Sine Fourier transforms of $f$,  respectively.
\end{Definicion}

\medskip

\begin{prop}
The $HK$-Fourier transform is well defined.
\end{prop}

\begin{proof}
 Suppose $f=u_1+v_1=u_2+v_2$ with $u_i \in \mathcal{L}^1(\mathbb{R})$ and $v_i\in BV_0(\mathbb{R})$ for $i=1, 2$. Therefore, 
$$u_1-u_2=v_2-v_1\in \mathcal{L}^1(\mathbb{R})\cap BV_0(\mathbb{R}).$$

This yields the result since the Henstock-Kurzweil integral coincides with the Lebesgue integral on the intersection $\mathcal{L}^1(\mathbb{R})\cap BV_0(\mathbb{R})$, see \cite{{M.Guadalupe}}. Therefore, $\mathcal{F}_{HK}(f)$ does not depend on the representation of the function $f$.

\end{proof}

Note that some integrals in (\ref{TransformadaHK1}) might not converge at $s=0$. 
Recently, in \cite{Alfredo}  was shown that the $HK$-Fourier Cosine transform is a bounded linear operator from $BV_{0}(\mathbb{R})$ into $HK(\mathbb{R})$. 
This is related to the question about whether the $HK$-Fourier 
transform is continuous at $s=0$.  
By \cite[Theorem 2.5.]{Francisco} the $HK$-Fourier Cosine and  Sine transforms of any function $f$ in $BV_0(\mathbb{R})$, $\mathcal{F}^{C}_{HK}(f)$ and  $\mathcal{F}^{S}_{HK}(f)$  are continuous functions (except at $s=0$) and vanish at infinity as $o(|s|^{-1})$.

\begin{thm}[\cite{M.Guadalupe}]\label{PlancherelHK}
If $f\in \mathcal{L}^{p}(\mathbb{R})\cap(\mathcal{L}^{1}(\mathbb{R})+BV_{0}(\mathbb{R}))$, for $1\leq p\leq 2$, then
\begin{equation*}
\mathcal{F}_{HK}(f)\in \mathcal{L}^q(\mathbb{R})\cap C_{\infty}(\mathbb{R}\setminus{\{0\}}),
\end{equation*}
where $p^{-1}+q^{-1}=1$. Moreover, 
\begin{equation*}
\mathcal{F}_{p}(f)(s)=\mathcal{F}_{HK}(f)(s),
\end{equation*}
almost everywhere. In particular, if $f\in BV_0(\mathbb{R})$, then  $$\mathcal{F}_{HK}(f)\in C_{\infty}(\mathbb{R}\setminus{\{0\}}).$$
\end{thm}

\section{An  approach of $\mathcal{F}_p$ via $\mathcal{F}_{HK}$} \label{sec:5}

According to the classical theory, it is not always possible to achieve a pointwise  expression of the Fourier transform operator $\mathcal{F}_{p}$ in Lebesgue's theory of integration, for $1<p\leq 2$. This is   because there exist $p$-integrable functions  that are not absolutely integrable. Nevertheless, there exist
functions belonging to $\mathcal{L}^{p}(\mathbb{R})\cap BV_{0}(\mathbb{R})\setminus\mathcal{L}^{1}(\mathbb{R})$; in accordance with  Theorem \ref{PlancherelHK}  we can apply the Henstock-Kurzweil integral in order to approach the Fourier transform operator $\mathcal{F}_p$  on  subsets of $\mathcal{L}^{p}(\mathbb{R})\setminus\mathcal{L}^{1}(\mathbb{R})$, for $1<p\leq 2$. 

The set of absolutely continuous functions over each compact interval is denoted by $AC_{loc}(\mathbb{R})$  \cite{Gordon,Leoni,Liflyand, Liflyand3, Talvila2}.

\begin{thm}\label{spaceLoc}
If $\phi \in \mathcal{L}^p(\mathbb{R}) \cap BV_{0}(\mathbb{R}) \cap AC_{loc}(\mathbb{R})$, for $1\leq p\leq 2$, then 
\begin{itemize}
\item [(i)] $\mathcal{F}_{HK} (\phi) \in C_{\infty}(\mathbb{R}  \setminus \{0\}).$
\item [(ii)]$\mathcal{F}_{HK} (\phi)(s)= \mathcal{F}_{p}(\phi)(s)$ a.e.

\item [(iii)] For every  $s\in\mathbb{R}\setminus\{0\}$, 
\begin{equation}\label{FHK}
\mathcal{F}_{HK}(\phi)(s)=-\frac{i}{s}\mathcal{F}_1(\phi')(s).
\end{equation}
\item [(iv)] Moreover, $$\left| \mathcal{F}_{HK}(\phi)(s) \right|\leq \frac{1}{\sqrt{2\pi}} \cdot \frac{1}{|s|}\parallel\phi'\parallel_{1}.$$
\end{itemize} 
\end{thm}

\begin{proof}
Let $\phi\in \mathcal{L}^2(\mathbb{R}) \cap BV_{0}(\mathbb{R}) \cap AC_{loc}(\mathbb{R})$. By Theorem  \ref{PlancherelHK} we get $(i)$ and $(ii)$. Note that $\mathcal{F}_{HK} (\phi)(s)$ is defined for any $s\not=0$, whereas $\mathcal{F}_p (\phi)(s)$ is defined almost  everywhere. Applying the Hake Theorem we get,
\begin{eqnarray}\label{meanvalue}
\mathcal{F}_{HK}(\phi)(s)=\frac{1}{\sqrt{2\pi}} \lim_{T\rightarrow\infty}\left[\int_{-T}^{T}\cos(st)\phi(t)dt- i\int_{-T}^{T}\sin(st)\phi(t)dt\right].
\end{eqnarray}
From the Multiplier Theorem \cite{Bartle} and the hypothesis for $\phi$  we have
\begin{eqnarray*}
\int_{\mathbb{R}} \cos(st)\phi(t)dt=-\lim_{T\rightarrow\infty} \int_{-T}^{T}\left[ \frac{\sin(st) +\sin(sT)}{s}\right]d\phi.
\end{eqnarray*}


\noindent Similarly for the Sine  Fourier transform we get
\begin{eqnarray*}\label{sin}
\int_{\mathbb{R}}\sin(st)\phi(t)dt=-\lim_{T\rightarrow\infty}\int_{-T}^{T} \left[\frac{-\cos(st)+\cos(sT)}{s}\right]d\phi.
\end{eqnarray*}
Where we have used that $\phi\in BV_0(\mathbb{R})$ vanishes at infinity.
This yields, from (\ref{meanvalue}),
\begin{eqnarray*}
\mathcal{F}_{HK}(\phi)(s)&=&
-\lim_{T \rightarrow \infty} \frac{1}{\sqrt{2\pi}} \int_{-T}^{T} \frac{\sin(st)+\sin(sT)}{s}d\phi \nonumber \\
&& +\  i\ \lim_{T \rightarrow \infty} \int_{-T}^{T}\frac{-\cos(st)+\cos(sT)}{s}d\phi .
\end{eqnarray*}
Note that  the  Stieltjes-type integrals below exist as Riemann-Stieltjes and Lebesgue-Stieltjes integrals  \cite{Ashordia,Groh,Monteiro}. Since  $\phi \in AC_{loc}(\mathbb{R})$,   by \cite[Theorem 6.2.12]{Monteiro} and \cite[Exercise 2, pag. 186]{Mc} it follows that
$$\int_{-T}^{T} (\sin(st)+\sin(sT))d\phi=  \int_{-T}^{T}\left(\sin(st) +\sin(sT)\right)\phi'(t)dt$$ 
and
$$  \int_{-T}^{T}(-\cos(st)+\cos(sT))d\phi 
=\int_{-T}^{T} \left(-\cos(st)+\cos(sT)\right) \phi'(t)dt.
$$
Since $\phi\in BV_0(\mathbb{R})\cap AC_{loc}(\mathbb{R})$, one gets

\begin{equation}\label{limite cero}
\lim_{T\rightarrow\infty}\left[(\cos(sT)-\sin(sT))\int_{-T}^{T}\phi'(t)dt\right]=0.
\end{equation}
and \cite[ Corollary 2.23]{Leoni} implies that $\phi'\in \mathcal{L}^{1}(\mathbb{R})$. Therefore, we get
\begin{eqnarray}\mathcal{F}_{HK}(\phi)(s)&=& \frac{1}{\sqrt{2\pi}}\frac{1}{s}\lim_{T \rightarrow \infty}\left[  -\int_{-T}^{T}[\sin(st)+i\cos(st)]\phi'(t)dt\right] \nonumber\\&=& \frac{1}{is}\mathcal{F}_1(\phi')(s).\nonumber
\end{eqnarray}
Furthermore,
 $$|\mathcal{F}_{HK}(\phi)(s)|\leq \frac{1}{\sqrt{2\pi}} \cdot \frac{1}{|s|}\parallel\phi'\parallel_{1}.$$
For $1\leq p<2$ the same formulas and argumentation are valid.   
\end{proof}
\begin{remark} Since $f\in BV_{0}(\mathbb{R}) \cap AC_{loc}(\mathbb{R})$, by \cite[Theorem 7.5, pag. 281]{Bartle} and \cite[Theorem 3.39]{Leoni} we have that $||\phi'||_1=Var(\phi,\mathbb{R})$ and $\phi\in AC(\mathbb{R})$. Thus, $  BV_{0}(\mathbb{R}) \cap AC(\mathbb{R})=  BV_{0}(\mathbb{R}) \cap AC_{loc}(\mathbb{R})$.
\end{remark}
From Theorem \ref{spaceLoc} we have the following result.

\begin{cor}\label{corollary1} Let  $\phi\in \mathcal{L}^p(\mathbb{R}) \cap BV_{0}(\mathbb{R}) \cap AC_{loc}(\mathbb{R})$, for $1\leq p\leq 2$.
\begin{itemize}
\item [(i)] If $\phi$ is an even function, then

\begin{equation}\label{even}\mathcal{F}_{HK}(\phi)(s)= -\sqrt{\frac{2}{\pi}}\cdot\frac{1}{s} \int_{0}^{\infty} \sin(st)\phi'(t)dt. 
\end{equation}

\item [(ii)] If $\phi$ is an odd function, then 
\begin{eqnarray}\label{odd}
\mathcal{F}_{HK}(\phi)(s)=-i\sqrt{\frac{2}{\pi}}\cdot \frac{1}{s}\int_{0}^{\infty}\left[ \cos(st)-1\right]\phi'(t)dt.
\end{eqnarray}
\end{itemize}
In either case, $\mathcal{F}_{HK}(\phi)(s)=\mathcal{F}_{p}(\phi)(s)$  a.e.,  where  $\mathcal{F}_{HK}(\phi)\in C_{\infty}(\mathbb{R}\setminus\{0\}).$
\end{cor}

E. Liflyand in \cite{Liflyand0}-\cite{Liflyand3} worked on a  subspace of $BV_{0}(\mathbb{R}) \cap AC_{loc}(\mathbb{R})$ to obtain integrability and asymptotic formulas for the Fourier  transform. We restrict to the domain of the Fourier transform operator  in order to provide   new integral expressions  of the classical Fourier transform $\mathcal{F}_p$.

The implications from these results are that the classical Fourier transform $\mathcal{F}_p(f)(s)$ for $f$ in a dense subspace  of ${\mathcal L}^p(\mathbb{R})$ is represented by a Lebesgue integral,  is a continuous function, except at $ s= 0$ and vanishes at infinity as $o(|s|^{-1})$.

\medskip

The algorithms of numerical integration are very important in applications, for example, approximation of the Fourier transform  have implications in digital image processing, economic estimates, acoustic phonetics, among others \cite{Bracewell}, \cite{Ruzhansky}. There exist  integrable functions whose primitives cannot be calculated explicitly; thus  numerical integration is fundamental to achieve explicit results.
 
 Also note that the Lebesgue integral is not suitable for numerical approximations.
 Alternatively, (\ref{even}) and (\ref{odd}) provide   expressions  that might be used to approximate numerically $\mathcal{F}_{p}(f)$   at specific values.  Actually, as a consequence of the Hake Theorem, it is possible to approximate  $\mathcal{F}_{HK}(\phi)(s)$ via the relation
\begin{equation}\label{aproximacion}\mathcal{F}_{HK}(\phi)(s)\approx\frac{1}{\sqrt{2\pi}} \int_{|t|\leq M}e^{-ist}\phi(t)dt \;\quad (M\to\infty),
\end{equation}
for any $s\neq 0$, $\phi\in \mathcal{L}^p(\mathbb{R})\cap BV_{0}(\mathbb{R})\setminus \mathcal{L}^1(\mathbb{R})$ ($1<p\leq 2$). Moreover,   Theorem \ref{spaceLoc}  justifies and
assures that $\mathcal{F}_p(\phi)(s)$ is asymptotically approximated by (\ref{aproximacion}). Note that Lebesgue's theory of integration only assures  convergence of the integrals in (\ref{aproximacion}) for a sequence of values of $M$ and $s$  in some (unknown) subset $\mathfrak{A}\varsubsetneq\mathbb{R}.$ 


\section{ Differentiability of the Fourier transform}

A classical theorem in Lebesgue's theory is about differentiability under the  integral sign \cite{Rudin} and \cite{Gasquet}. The following result is a generalization.
\begin{thm}\label{derivation}  Let $f \in \mathcal{L}^1(\mathbb{R})+BV_{0}(\mathbb{R})$ such that $g(t):=t f(t)$ belongs to $\mathcal{L}^1(\mathbb{R})+BV_{0}(\mathbb{R})$.   Then  $\mathcal{F}_{HK}(f)$ is \textcolor{red}{continuously} differentiable away from zero and 



 \begin{equation}
\label{derivative od FKHFT} \frac{d}{ds}\mathcal{F}_{HK}(f)(s)=-i\mathcal{F}_{HK}(g)(s), \quad (s\neq 0).
\end{equation}

\end{thm}
\begin{proof}
 For the case 
$$f, g \in BV_{0}(\mathbb{R}),$$
let us define  $G(s,t):=\cos(st)f(t)$  in $\mathcal{L}^{1}_{loc}(\mathbb{R})$ with respect to $s$ for all $t\in \mathbb{R}$, where $[\alpha,\beta]$ is any compact interval such that $0\not\in [\alpha,\beta]$ and 
 let us consider the sequence $(\Phi_{n}) $ where 
\begin{equation}\label{Phin}
\Phi_{n}(s):=\int_{-n}^{n}\frac{d}{ds} G(s,t) dt
\end{equation} 
with $n\in \mathbb{N}$. Since 
\begin{equation*}
|\Phi_{n}(s)|\leq \frac{2 }{|s|}Var(g,\mathbb{R} ),
\end{equation*} 
then  $\Phi_{n}\subset L^{1}[\alpha,\beta]$, where $[\alpha,\beta]$ is any compact interval  such that $0\not\in [\alpha,\beta]$. Applying the Dominated Convergence Theorem, Fubini's Theorem and Hake's Theorem \cite{Bartle}, we get

\begin{eqnarray}\label{integral FS}
 \frac{1}{\sqrt{2\pi}}\int_{\alpha}^{\beta}\int_{-\infty}^{\infty}\frac{d}{ds} G(s,t)\; dt\; ds&=&
 \frac{1}{\sqrt{2\pi}}\int_{\alpha}^{\beta}\int_{-\infty}^{\infty}-\sin(st)g(t)\;dt\;ds\nonumber\\
&=& \frac{1}{\sqrt{2\pi}}\int_{\alpha}^{\beta}\lim_{n\rightarrow\infty} \Phi_{n}(s)\; ds\\
&=& \frac{1}{\sqrt{2\pi}} \lim_{n\rightarrow\infty} \int_{-n}^{n} f(t)\Big[\cos(\beta t)-\cos(\alpha t) \Big]dt\nonumber\\
&=& \mathcal{F}_{HK}^{C}(f)(\beta)-\mathcal{F}_{HK}^{C}(f)(\alpha).\nonumber
\end{eqnarray}
On the other hand,
\begin{eqnarray}
  \frac{1}{\sqrt{2\pi}}\int_{-\infty}^{\infty}\int_{\alpha}^{\beta} \frac{d}{ds} G(s,t)\; ds\; dt&=&
\frac{1}{\sqrt{2\pi}}\int_{-\infty}^{\infty} \cos(\alpha t)f(t)-\cos(\beta t)f(t)dt\nonumber\\&=&
\mathcal{F}_{HK}^{C}(f)(\beta)-\mathcal{F}_{HK}^{C}(f)(\alpha).\label{1} 
\end{eqnarray}
From (\ref{integral FS}), (\ref{1}) and  \cite[Theorem 4]{Talvila2},  we get that the $HK$-Cosine Fourier  transform is differentiable under the integral  sing. 
Since  $g\in BV_0(\mathbb{R})$, by Theorem \ref{PlancherelHK}, $\mathcal{F}_{HK}^{C}(f)'$ is a continuous function (except at $s=0$) vanishing at infinity. By similar arguments,  $\mathcal{F}_{HK}^{S}(f)'(s)=\mathcal{F}_{HK}^{C}(g)(s)$ for any $s\neq 0$. 

For the general case, we suppose $t f=g_1+g_2\in \mathcal{L}^1(\mathbb{R})+BV_{0}(\mathbb{R})$. Then (\ref{Phin}) with $g=g_1+g_2$  obeys also $\{\Phi_n\}\subset  L^1[\alpha,\beta],$ so that (\ref{integral FS}) remains valid.  Therefore, the $HK$-Fourier transform is differentiable  and (\ref{derivative od FKHFT}) is obtained. 
\end{proof}

\begin{cor}
Under the assumptions of Theorem \ref{derivation}.  Then
 $$\mathcal{F}^{S}_{HK}(f)\in ACG^{*}_{loc}(\mathbb{R}).$$
\end{cor}
\begin{proof}  
Theorem \ref{derivation} implies that $\mathcal{F}^{S}_{HK}(f)$  is a continuously differentiable function away from zero, and \cite[Corollary 1]{Alfredo} yields that its derivate is actually a function in $HK_{loc}(\mathbb{R})$. Therefore  \cite[Theorem 2]{Talvila2} gives the result.  
\end{proof}

We will extend this theorem to study differentiability of the Fourier transform $\mathcal{F}_p(f)$ for $1\leq p\leq 2$.
\begin{lemma}\label{lemma1}
Suppose $1\leq p\leq 2$ and $f\in \mathcal{L}^p(\mathbb{R})$. Then there exists a subsequence $(n_k)\in\mathbb{N}$ such that
\begin{equation}\label{limit}
\frac{1}{\sqrt{2\pi}}  \lim_{k\to\infty}\int_{-n_k}^{n_k} e^{-is x}f(x) dx=\mathcal{F}_p(f)(s),\end{equation}
almost everywhere on $\mathbb{R}$.
\end{lemma}
\begin{proof} The cases $p=1$ or $p=2$ follow from \cite{Reed, Rudin}. For $1<p<2$,    due to {\cite{Grafakos,M.Guadalupe, Pinsky}} there exist functions  $f_1\in \mathcal{L}^1(\mathbb{R})\cap \mathcal{L}^p(\mathbb{R}), \ f_2\in \mathcal{L}^2(\mathbb{R})\cap \mathcal{L}^p(\mathbb{R})$ such that $f=f_1+f_2$. It follows that
$$\mathcal{F}_p(f)=\mathcal{F}_1(f_1)+\mathcal{F}_2(f_2).$$
Applying once again \cite{Reed, Rudin}, we obtain a sequence $ (n_k)\subset \mathbb{N}$ such that
$$ \frac{1}{\sqrt{2\pi}}\lim_{k\to\infty}\int_{-n_k}^{n_k} e^{-is x}f(x) dx=\mathcal{F}_2(f_2)(s)$$
almost everywhere in $\mathbb{R}$. This yields,
\begin{eqnarray}
\frac{1}{\sqrt{2\pi}}\lim_{k\to\infty}\int_{-n_k}^{n_k} e^{-is x}f(x) dx &=& \frac{1}{\sqrt{2\pi}}\lim_{k\to\infty}\int_{-n_k}^{n_k} e^{-is x}f_1(x) dx\nonumber  \\ &+& \frac{1}{\sqrt{2\pi}}\lim_{k\to\infty}\int_{-n_k}^{n_k} e^{-is x}f_2(x) dx\nonumber\\
&=&\mathcal{F}_1(f_1)(s)+\mathcal{F}_2(f_2)(s)\nonumber  \\
&=&\mathcal{F}_p(f)(s),\nonumber
\end{eqnarray}
almost everywhere. This proves the statement.
\end{proof}
Below we use the notation $p^{-1}+q^{-1}=1$.

\begin{prop}\label{Propo1}
Let $1\leq p\leq 2$ be fixed. If $f\in \mathcal{L}^p(\mathbb{R})$ and $g(t):=tf(t)$ belongs to $\mathcal{L}^1(\mathbb{R})+ BV_0(\mathbb{R})$,  then by redefining   $\mathcal{F}_{p}(f)(s)$ on a set of measure zero, it yields 
\begin{equation*}
\label{derivative od FPFT} \frac{d}{ds}\mathcal{F}_{p}(f)(s)=-i\mathcal{F}_{HK}(g)(s),  \quad ( s\not=0).
\end{equation*}
\end{prop}

\begin{proof} Take values  $s=\alpha,$ and $s=\beta$ such that (\ref{limit})  is valid. Suppose that $0<\alpha<\beta$. 
Proceeding similarly as in Theorem \ref{derivation}, we have 
\begin{eqnarray}
-i\int_{\alpha}^{\beta}\mathcal{F}_{HK}(g)(s)ds&=& 
 \int_{\alpha}^{\beta} \lim_{k\rightarrow\infty} \frac{ -i}{\sqrt{2\pi}}  \int_{-n_{k}}^{n_{k}}e^{-ist}tf(t)dtds \nonumber\\
&=&  \frac{-i}{\sqrt{2\pi}}\lim_{k\rightarrow\infty}\int_{-n_{k}}^{n_{k}}\int_{\alpha}^{\beta}e^{-ist}tf(t)dsdt \nonumber\\
&=& \frac{-i}{\sqrt{2\pi}}\lim_{k\rightarrow\infty}\int_{-n_{k}}^{n_{k}} (e^{-i\beta t}-e^{-i\alpha t})f(t))dt \nonumber\\
&=&\mathcal{F}_{p}(f)(\beta)-\mathcal{F}_{p}(f)(\alpha),\label{Fp}
\end{eqnarray}
where (\ref{Fp}) holds almost everywhere by Lemma \ref{lemma1}. This implies the statement of the proposition.
\end{proof}

\begin{cor} Assume $f\in \mathcal{L}^{p}(\mathbb{R})$ and $tf\in\mathcal{L}^p(\mathbb{R})\cap BV_0(\mathbb{R})$.  Then, by redefining  $\mathcal{F}_{p}^{S}(f)$  on a set of measure zero yields 
 
 $$\mathcal{F}_{p}^{S}(f)\in ACG^{*}(\mathbb{R}).$$
\end{cor}
\begin{proof}
This follows from Proposition \ref{Propo1},  \cite[Corollary 1]{Alfredo} and  \cite[Theorem 2]{Talvila2}.
\end{proof}


\begin{prop}\label{derivada in Lp} Let $1< p\leq 2$ be fixed, $f\in \mathcal{L}^p(\mathbb{R})$ and $tf=h_{p}+h_{0} \in \mathcal{L}^{p}(\mathbb{R})+BV_0(\mathbb{R})$.  Then, by redefining  $\mathcal{F}_{p}(f)(s)$ on a set of measure zero, it yields  
$$\mathcal{F}_{p}(f) \in AC_{loc}(\mathbb{R}\setminus \{0\})\cap \mathcal{L}^{q}(\mathbb{R})$$ and
\begin{equation*}
\label{derivative od FPFT} \frac{d}{ds}\mathcal{F}_{p}(f)(s)=-i[\mathcal{F}_{p}(h_{p})(s)+\mathcal{F}_{HK}(h_{0})(s)], \quad a.e.
\end{equation*}

\end{prop}


\medskip
\begin{proof}

 Due to $$ \frac{1}{\sqrt{2\pi}}\int_{-n}^{n}e^{-ist}h_{p}(t) dt\rightarrow\mathcal{F}_{p}(h_{p}) \, \, (n\rightarrow\infty),$$ there exists $M>0$ such that
$$\Big\| \frac{1}{\sqrt{2\pi}}\int_{-n}^{n}e^{-ist}h_{p}(t) dt \Big\|_{q}\leq M<\infty $$ uniformly on $n\in \mathbb{N}$. As argued in equation (\ref{integral FS}), 


\begin{eqnarray*}
-i\int_{\alpha}^{\beta}\mathcal{F}_{p}(h_{p})(s)+\mathcal{F}_{HK}(h_{0})(s)ds&=& 
\frac{-i}{\sqrt{2\pi}}\left[\lim_{k\rightarrow\infty} \int_{-n_k}^{n_k}\int_{\alpha}^{\beta} e^{-ist}h_{p}(t) dsdt + \lim_{k\rightarrow\infty} \int_{-n_k}^{n_k}\int_{\alpha}^{\beta} e^{-ist}h_{0}(t) dsdt\right]\nonumber \\
& =&\frac{-i}{\sqrt{2\pi}}\lim_{k\rightarrow\infty} \int_{-n_k}^{n_k}\int_{\alpha}^{\beta} e^{-ist}(h_{p}+h_{0})(s)dsdt\nonumber\\
&=&\frac{-i}{\sqrt{2\pi}} \lim_{k\rightarrow\infty} \int_{-n_k}^{n_k}\int_{\alpha}^{\beta} e^{-ist}tf(t)dsdt\nonumber\\
&=&\frac{1}{\sqrt{2\pi}} \lim_{k\rightarrow\infty} \int_{-n_k}^{n_k}( e^{-i\beta t}-e^{-i\alpha t} )f(t)dt\nonumber \\
&=& \mathcal{F}_{p}(f)(\beta)-\mathcal{F}_{p}(f)(\alpha).
\end{eqnarray*}
Where we take a subsequence of $(n_{k})$, if neccessary. Here  $\alpha$, $\beta$ are values such that (\ref{limit}) is valid. 

\end{proof}


\begin{cor}\label{Coro2} Assume the hypothesis of Proposition \ref{derivada in Lp}.  Then, by redefining  $\mathcal{F}^{S}_{p}(f)$ on a set of measure zero      
 $$\mathcal{F}_{p}^{S}(f)\in ACG_{loc}^*(\mathbb{R}).$$
\end{cor}
\begin{proof}
Similar arguments as above give the result.
\end{proof}

\bigskip

Now we show some examples. 

\bigskip

\begin{example}\label{example 1}
Let  $\phi (t):=(1+t^2)^{-1}$ on $[0,\infty)$ and zero otherwise. It is easy to see that $\phi$ belongs to $ \mathcal{L}^1(\mathbb{R})\cap \mathcal{L}^2(\mathbb{R})$. Moreover,  $g(t)=t\phi(t)$ belongs to $BV_0(\mathbb{R})\setminus \mathcal{L}^{1}(\mathbb{R})$.  By   Proposition \ref{Propo1} we have
$$\frac{d}{ds}\mathcal{F}_{2}(\phi)(s)=\frac{d}{ds}\mathcal{F}_{1}(\phi)(s)=-i\mathcal{F}_{HK}(g)(s)\quad (s\not= 0). $$

\end{example}

\bigskip


\begin{example}
 Let $\phi(t) = \arctan |t|-\frac{\pi}{2}$. Note that $\phi\in BV_{0} (\mathbb{R})\cap\mathcal{L}^2(\mathbb{R})\setminus \mathcal{L}^1(\mathbb{R})$.  However, $t\phi(t)$ does not belong to $\mathcal{L}^{1}(\mathbb{R})+BV_0(\mathbb{R})$.
 By Corollary \ref{corollary1} we have that
\begin{equation*}
\mathcal{F}_{2}(\phi)(s)=  -\sqrt{\frac{2}{\pi}}\cdot \frac{1}{s}\mathcal{F}_{1}^S(\tau ' )(s),
\end{equation*}
where
$\tau(t)=\phi\cdot \chi_{(0,\infty)}(t)$. 
Note that $\tau'\in \mathcal{L}^1(\mathbb{R})\cap\mathcal{L}^2(\mathbb{R})$, hence $\mathcal{F}_{1}^S(\tau')=\mathcal{F}_{2}^S(\tau')$.  Applying  Proposition \ref{Propo1} to $\mathcal{F}_{2}^S(\tau')$ we have that 
$$\frac{d}{ds}\mathcal{F}_{2}^S(\tau')(s)=\mathcal{F}_{HK}^C(g)(s),$$
here $g(t)=
t\tau ' (t) \in BV_0(\mathbb{R})$.
Thus,  $\mathcal{F}_{2}(\phi)$  is a continuously differentiable function away from zero
and  $$\frac{d}{ds}\mathcal{F}_{2}(\phi)(s)=\sqrt{\frac{2}{\pi}}\left[\frac{1}{s^2}\mathcal{F}_{1}^S(\tau ')(s)-\frac{1}{s}\mathcal{F}_{HK}^C(g)(s)\right], \; \; \text{a.e.}$$


\end{example}

\bigskip

\begin{example}Let $h_{1}(t):= 1- C(2/\pi \arctan(t))$ and $ h_{2}(t):= \sqrt[3]{t}  \sin(1/t)$,  where $ C$ is the Cantor function \cite{Corothers} . We take
$$f(t) = \left\{ \begin{array}{ll} t ^{-1}\cdot h_{1}(t) & \text{if\ }\;\; t>2, \\ h_{2}(t)& \text{if\ }\;\;0<t<1,\\
0& \text{otherwise}. 
\end{array} \right.$$
Due to $ t^{-1}\cdot h_{1}(t)$ belongs to $\mathcal{L}^p([2,\infty))$ for  $p\geq 1$, it follows that $f\in \mathcal{L}^{p}(\mathbb{R})$, see \cite{Russell}. Moreover, 
 $$h_{1} \in BV_{0}([2,\infty))\setminus \mathcal{L}^{1}([2,\infty)).$$ 
In addition, $g(t)=t\cdot f(t)$ is not in $\mathcal{L}^1(\mathbb{R}) $, but in $BV_0(\mathbb{R})+ \mathcal{L}^{1}(\mathbb{R})$. Applying Proposition \ref{Propo1},  we have that for $1\leq p\leq 2,$
\begin{eqnarray*}
\frac{d}{ds}\mathcal{F}_p(f)(s)=-i\mathcal{F}_{HK}(g)(s)=-i[\mathcal{F}_{HK}(g_1)(s)+\mathcal{F}_1(g_2)(s)],\;\; (s\neq 0),
\end{eqnarray*}
where $g_1(t):=h_1\chi_{(2,\infty)}(t)$ and $g_2(t):=t\cdot h_2\chi_{(0,1)}(t)$.
\end{example}




\section{Conclusions} \label{sec:6}


An integral  representation  of the Fourier transform  is obtained on the subspace  $ \mathcal{L}^p(\mathbb{R}) \cap BV_{0}(\mathbb{R})\cap AC_{loc}(\mathbb{R})\setminus \mathcal{L}^1(\mathbb{R})$, for $1<p\leq 2$. This is possible by switching to the Henstock-Kurzweil integral. Furthermore, expressions (\ref{even}) and (\ref{odd}) give explicit formulas of $\mathcal{F}_{p}$  over that subspace.  
 Using our results, specific values of the Fourier transform of  particular functions might be approximated   with arbitrary accuracy. Moreover, it was shown  differentiability of  $\mathcal{F}_{p}(\phi)(s)$ by extending a classical theorem in Lebesgue's theory. This illustrates the applicability of the results obtained, which are original in Fourier Analysis over $\mathcal{L}^p(\mathbb{R})$.
  
  
\section*{Acknowledgements}\label{sec:Acknowledgements}
The first author is supported by GA 20-11846S of the Czech Science Foundation. J.H.A. and M.B. acknowledge  partial support from CONACyT--SNI.

 \bigskip \smallskip

 \it

 \noindent
J. H. Arredondo\\
Departamento de Matem\'aticas,\\ Universidad Aut\'onoma Metropolitana - Iztapalapa\\
Av. San Rafael Atlixco 186, M\'exico City, 09340, M\'exico.\\
e-mail: iva@xanum.uam.mx\\ 
\\[12pt]
M. G. Morales\\
Department of Mathematics and Statistics,\\ 
Faculty of Science, Masaryk University,\\
 Kotl\'a\v{r}sk\'a 2, 611 37 Brno, Czech Republic.\\
e-mail: maciasm@math.muni.cz\\
\\[12pt]
M. Bernal\\
Departamento de Matem\'aticas,\\ Universidad Aut\'onoma Metropolitana - Iztapalapa\\
Av. San Rafael Atlixco 186, M\'exico City, 09340, M\'exico.\\
e-mail: mbg@xanum.uam.mx 
\end{document}